\documentclass[11pt]{amsart}
\pdfoutput=1
\usepackage[usenames,dvipsnames]{xcolor}
\usepackage[utf8]{inputenc}
\usepackage[T1]{fontenc}
\usepackage[letterpaper,margin=1in]{geometry}
\usepackage{lmodern}
\usepackage{microtype}
\usepackage{etoolbox}
\usepackage{mathtools}
\usepackage{amssymb}
\usepackage{enumitem}
\usepackage[round]{natbib}
\usepackage{hyperref}

\numberwithin{equation}{section}
\patchcmd{\section}{\scshape}{\fontfamily{ppl}\bfseries\scshape\selectfont}{}{}
\hypersetup{colorlinks=true,linkcolor=blue!45!black,citecolor=blue!45!black,urlcolor=red!65!black}
\setlist{nosep,leftmargin=*}

\newtheorem{theorem}{Theorem}[section]
\newtheorem{prop}[theorem]{Proposition}
\newtheorem{lemma}[theorem]{Lemma}
\newtheorem{rmk}[theorem]{Remark}
\theoremstyle{definition}
\newtheorem{defn}[theorem]{Definition}

\newcommand{\reals}{\mathbb{R}}
\newcommand{\E}{\mathbb{E}}
\newcommand{\nlsum}{\sum\nolimits}
\DeclareMathOperator{\diag}{diag}
\DeclareMathOperator{\per}{per}

\DeclareMathAlphabet{\mathbbold}{U}{bbold}{m}{n}
\newcommand{\ii}{\mathbbold{i}}   %

\DeclareMathOperator{\Cov}{Cov}
\DeclareMathOperator{\Var}{Var}

\newcommand{\one}{\mathbf 1}

\newcommand{\cE}{\mathcal E}
\newcommand{\cZ}{\mathcal Z}
\newcommand{\T}{\mathbb T}
\newcommand{\dotast}{\mathbin{\dot\ast}}

\hypersetup{
  pdftitle={A Koteljanskii inequality for permanents},
  pdfauthor={Suvrit Sra}
}

\setlist[itemize]{topsep=4pt,itemsep=2pt,parsep=2pt}
\setlist[enumerate]{topsep=4pt,itemsep=2pt,parsep=2pt}

\begin{document}
\title{A Koteljanskii inequality for permanents}
\author[Suvrit Sra]{Suvrit Sra\textsuperscript{\normalfont*}}
\thanks{*\,TU Munich, Dept.~of Mathematics, School of CIT; Garching, Germany.
  Email: \texttt{s.sra@tum.de}.}
\begin{abstract}
We prove a permanental analogue of Koteljanskii's inequality. If \(A\) is an inverse \(M\)-matrix that becomes symmetric after a positive diagonal similarity, then \(\per(A_{S\cup T})\per(A_{S\cap T})\ge\per(A_S)\per(A_T)\) for all \(S,T\subseteq[n]\), where \(A_S\) is the principal submatrix indexed by \(S\). The proof expresses permanents as moments of a complex Gaussian vector and uses Ginibre's correlation inequality. Finally, an explicit counterexample shows that symmetry cannot be dropped. %
\end{abstract}
\maketitle

\vspace*{-5pt}
\section{Introduction}\label{sec:intro}
For an \(n\times n\) matrix \(A\) and \(S\subseteq[n]\), write \(A_S=A[S,S]\) for the principal submatrix with row and column set \(S\), and set \(\det(A_\varnothing)=\per(A_\varnothing)=1\). Koteljanskii's inequality \citep{koteljanskii1950theory,koteljanskii1953property}
\[
  \det(A_{S\cup T})\det(A_{S\cap T})\le\det(A_S)\det(A_T),
  \qquad S,T\subseteq[n],
\]
holds for positive semidefinite matrices, and also for \(M\)-matrices and inverse \(M\)-matrices \citep{fallat1998characterization}. Its case \(S\cap T=\varnothing\) is Fischer's inequality, and iterating Fischer's inequality gives Hadamard's inequality \(\det(A)\le\prod_ia_{ii}\). For positive semidefinite matrices the permanent reverses Hadamard's inequality, \(\per(A)\ge\prod_ia_{ii}\) \citep{marcus1963permanent}, and Lieb's permanental inequality \citep{lieb1966} reverses Fischer's: \(\per(A_{S\cup T})\ge\per(A_S)\per(A_T)\) for disjoint \(S,T\). These two reversals suggest the permanental analogue of Koteljanskii's inequality,
\begin{equation}\label{eq:lsm}
  \per(A_{S\cup T})\per(A_{S\cap T})\ge \per(A_S)\per(A_T),
  \qquad S,T\subseteq[n],
\end{equation}
where as usual the permanent of an $n\times n$ matrix is defined as
\begin{equation}
  \label{eq:1}
  \per(A)=\sum_{\pi\in \mathrm{Sym}_n}\prod\nolimits_{i\in [n]}A_{i,\pi(i)}.
\end{equation}

\vskip4pt
\noindent\emph{Which matrices satisfy~\eqref{eq:lsm}}? %
Contrary to the determinant, here positive semidefiniteness does not suffice, even with nonnegative entries: the positive semidefinite matrix \(A\) with rows \((1,1,0)\), \((1,2,1)\), \((0,1,1)\) has \(\per(A)\per(A_{\{2\}})=8<9=\per(A_{\{1,2\}})\per(A_{\{2,3\}})\). This obstruction led us to consider inverse \(M\)-matrices instead, as they also satisfy Koteljanskii's inequality. Numerous initial experiments in early 2015 convinced us that for inverse $M$-matrices the inequality~\eqref{eq:lsm} could hold. %

\vskip5pt
A matrix \(A\) is an inverse \(M\)-matrix if \(A\ge0\) entrywise, \(A\) is invertible, and \(A^{-1}\) has nonpositive off-diagonal entries. We prove~\eqref{eq:lsm} for inverse \(M\)-matrices that are symmetric after a positive diagonal similarity, which we call \emph{symmetrizable}, and show that~\eqref{eq:lsm} can fail without this hypothesis.

\begin{theorem}[Koteljanskii for $\per$]\label{thm:main}
Let \(A\) be a symmetrizable inverse \(M\)-matrix. Then the set function \(S\mapsto\per(A_S)\) is log-supermodular, i.e., inequality~\eqref{eq:lsm} holds for all \(S,T\subseteq[n]\).
\end{theorem}

\begin{theorem}[Counterexample]\label{thm:counterexample}
  The matrix
  \begin{small}
    \[
      A=\begin{pmatrix}
        13&1&7&7&7\\
        1&13&7&7&7\\
        1&1&7&7&7\\
        1&1&1&7&7\\
        1&1&1&1&7
      \end{pmatrix}
    \]
  \end{small}
is a non-symmetrizable inverse \(M\)-matrix, and \(\per(A)\per(A_{\{3,4,5\}})<\per(A_{\{1,3,4,5\}})\per(A_{\{2,3,4,5\}})\).
\end{theorem}
For the proof of Theorem~\ref{thm:main} it suffices to consider symmetric (and hence, also positive definite) inverse $M$-matrices since diagonal similarity does not change principal permanents. Our starting point is the classical representation of \(\per(A_S)\) as the moment \(\E\prod_{i\in S}|\zeta_i|^2\) of a complex Gaussian vector.  %
Then, Ginibre's inequality \citep{ginibre1970general} helps us show that the density of \(X=(|\zeta_1|^2,\dots,|\zeta_n|^2)\) %
is MTP\(_2\) (multivariate totally positive of order two)---see Proposition~\ref{prop:radial}. This property provides a key building block for a tilting based argument that closes the proof. %
Our results leave open which non-symmetrizable inverse \(M\)-matrices satisfy~\eqref{eq:lsm}; Section~\ref{sec:open} lists this and other questions.

\subsection{Related work.} \citet{johnson1982inverse} surveys inverse \(M\)-matrices, with a sequel by \citet{johnson2011inverse}; for permanental inequalities on totally positive matrices see \citet{skandera2025permanental}. The MTP property and the FKG inequality for densities are due to \citet{karlin1980classes}. \citet{eisenbaum2014characterization} characterized positively correlated squared Gaussian vectors and extended the characterization to the permanental vectors of Section~\ref{sec:alpha}; see the erratum \citep{eisenbaum2022erratum} and \citet{finner2024positive}, who show that positively associated squared Gaussian vectors need not be MTP\(_2\) and extend this to multivariate $\chi$-square distributions. Positive association alone does not give~\eqref{eq:lsm}, which needs association after the tilt by \(\prod_{i\in S\cap T}X_i\) (Remark~\ref{rmk:assoc}); the example we provide therein also contradicts Corollary~11 of the technical report of \citet{lovasz2003}.

\subsection{Notation and definitions}\label{sec:prelim}
Throughout, we use the shorthand \([n]=\{1,2,\dots,n\}\). The central class of matrices for this paper is that of inverse $M$ matrices; we define these below.
\begin{defn}[\(M\)-matrices]
A square matrix \(M=(m_{ij})\) is a nonsingular \emph{\(M\)-matrix} if: (i) \(m_{ij}\le0\) for all \(i\ne j\); and (ii) \(M\) is invertible with \(M^{-1}\ge0\) entrywise. A matrix $A$ is an \emph{inverse $M$-matrix} if \(A=M^{-1}\) for some nonsingular \(M\)-matrix \(M\). %
\end{defn}
\noindent Let us also formally define log-supermodularity.
\begin{defn}[Log-supermodularity]
A set function \(f:2^{[n]}\to\reals_{\ge0}\) is \emph{log-supermodular} if
\[
  f(S\cup T)f(S\cap T)\ge f(S)f(T),\qquad S,T\subseteq[n].
\]
\end{defn}

We are now ready to prove our main result, Theorem~\ref{thm:main}.

\section{Proof of Theorem~\ref{thm:main}}\label{sec:proof}
The proof has four steps: Step~1 reduces to symmetric \(A\); Step~2 writes \(\per(A_S)\) as a squarefree moment of \(X=(|\zeta_1|^2,\dots,|\zeta_n|^2)\), where \(\zeta\) is a complex Gaussian vector with covariance \(A\); Step~3 shows that \(X\) has an MTP\(_2\) density, and once the phases are integrated out, this reduces to a correlation inequality of Ginibre;  Step~4 finishes by observing that squarefree moments of an MTP\(_2\) density are log-supermodular.

\subsection{Step 1, symmetry} Let \(\Delta\) be a positive diagonal matrix such that \(A'=\Delta A\Delta^{-1}\) is symmetric. This is again an inverse \(M\)-matrix: \(A'\ge0\), and \((A')^{-1}=\Delta M\Delta^{-1}\) has the sign pattern of \(M=A^{-1}\). Moreover, for every \(S\subseteq[n]\) and every permutation \(\pi\) of \(S\), $\prod_{i\in S}A'_{i,\pi(i)}
 =\prod_{i\in S}\frac{\Delta_{ii}}{\Delta_{\pi(i),\pi(i)}}A_{i,\pi(i)}
 =\prod_{i\in S}A_{i,\pi(i)}$. Thus, \(\per(A'_S)=\per(A_S)\); so henceforth we may assume that \(A\) is symmetric.

\subsection{Step 2, Gaussian moments} Now that $A$ is symmetric, we also need it to be positive definite so that we can use it as the covariance matrix of a complex Gaussian vector. Fortunately, a symmetric inverse \(M\)-matrix is automatically positive definite. Indeed, put \(x=A\one\); then \(x>0\), since no row of the invertible nonnegative matrix \(A\) vanishes. Since \(Mx=\one\), the symmetric matrix \(\diag(x)M\diag(x)\) has nonpositive off-diagonal entries and \(i\)-th row sum \(x_i>0\). It is therefore strictly diagonally dominant with positive diagonal, hence positive definite, and so are \(M\) and \(A\). Let \(\zeta\in\mathbb C^n\) be a complex Gaussian vector with covariance \(A\) and density \(\pi^{-n}\det M\,e^{-z^*Mz}\). Let \(X=(|\zeta_1|^2,\ldots,|\zeta_n|^2)\). By the moment theorem for complex Gaussian vectors \citep{reed1962moment}, \(\E\prod_{i\in S}X_i=\E\prod_{i\in S}\zeta_i\overline{\zeta_i}=\per(A_S)\); in particular, all squarefree moments of \(X\) are finite.

\subsection{Step 3, MTP\(_2\) density} Next, we wish to show that the density of $X$ is MTP${}_2$. Recall that for a positive density \(p\) on \((0,\infty)^n\), the MTP\(_2\) property is
\begin{equation}
  \label{eq:2}
  p(x\vee y)p(x\wedge y)\geq p(x)p(y).
\end{equation}
For $C^2$ densities~\eqref{eq:2} is equivalent to non-negative Hessian off-diagonals, i.e., $\partial_{ij}\log p\geq0$ for \(i\ne j\).

We prove this non-negativity for the density of $X$. To that end, first write \(M=D-C\), where \(D=\diag(d_1,\dots,d_n)\) with \(d_i=M_{ii}>0\), and \(C\ge0\) is symmetric with zero diagonal~\citep[Chapter~6]{berman1994}. Let \(\cE\) be the pairs \(\{i,j\}\) with \(C_{ij}>0\); using polar coordinates \(z_i=\sqrt{x_i}e^{\ii\theta_i}\), we observe that
\[
 z^*Mz=\nlsum_id_ix_i-\nlsum_{\{i,j\}\in\cE}2C_{ij}\sqrt{x_ix_j}\cos(\theta_i-\theta_j).
\]
Now integrate out \(\theta\) to obtain the density
\begin{equation}\label{eq:density}
 p(x)=\det M\exp\!\left(-\nlsum_i d_ix_i\right)\cZ(J(x)),
 \qquad J_{\{i,j\}}(x)=2C_{ij}\sqrt{x_ix_j},
\end{equation}
where \(\cZ\) is the partition function on $\cE$ defined as follows: For \(e=\{i,j\}\in\cE\) set \(\omega_e(\theta)=\cos(\theta_i-\theta_j)\), where \(\theta\in\T^n=(\reals/2\pi\mathbb Z)^n\). For nonnegative couplings \(J_e\), define
\[
 \cZ(J) :=\int_{\T^n}\exp\!\left(\nlsum_{e\in\cE}J_e\omega_e(\theta)\right)d\theta,
 \qquad
 d\mu_J(\theta)=\cZ(J)^{-1}
 \exp\!\left(\nlsum_{e\in\cE}J_e\omega_e(\theta)\right)d\theta,
\]
with normalized Haar measure \(d\theta\). Write \(\E_J\) and \(\Cov_J\) for expectation and covariance under \(\mu_J\); differentiation under the integral then gives
\[
 \frac{\partial\log\cZ}{\partial J_e}=\E_J[\omega_e],
 \qquad
 \frac{\partial^2\log\cZ}{\partial J_e\partial J_f}
 =\Cov_J(\omega_e,\omega_f).
\]
The sign of \(\partial_{ij}\log p\) is therefore controlled by the angular means \(\E_J[\omega_e]\) and covariances \(\Cov_J(\omega_e,\omega_f)\).

Inequality~\eqref{eq:3} below is a case of Ginibre's inequality \citep{ginibre1970general} for edge observables. We include its proof to make explicit which correlation property is used.
\begin{lemma}\label{lem:angular}
For every \(e,f\in\cE\), $\E_J[\omega_e]\geq0$, and $\Cov_J(\omega_e,\omega_f)\geq0$.
\end{lemma}
\begin{proof}
The absolutely convergent Fourier expansion $e^{J\cos u} =\sum_{p,q\geq0}\frac{(J/2)^{p+q}}{p!q!}e^{\ii(p-q)u}$ has nonnegative coefficients. The product over edges has the same property, so its integral against \(\cos(\theta_i-\theta_j)\) is nonnegative. Division by \(\cZ(J)>0\) proves the first inequality.

For the covariance, introduce two independent angular variables \(\theta,\phi\). With \(H(\theta)=\sum_gJ_g\omega_g(\theta)\),
\[
 2\cZ(J)^2\Cov_J(\omega_e,\omega_f)
 =\int (\omega_e(\theta)-\omega_e(\phi))(\omega_f(\theta)-\omega_f(\phi))
 e^{H(\theta)+H(\phi)}\,d\theta\,d\phi.
\]
The map \((u,v)\mapsto(u+v,u-v)\) is a surjective homomorphism of \(\T^{2n}\), and therefore pushes normalized Haar measure to normalized Haar measure. Orient each edge arbitrarily and write \(u_e=u_i-u_j\) and \(v_e=v_i-v_j\). Under the substitution \(\theta=u+v\), \(\phi=u-v\),
\[
 \omega_e(\theta)-\omega_e(\phi)=-2\sin u_e\sin v_e,
 \qquad
 H(\theta)+H(\phi)=2\sum_gJ_g\cos u_g\cos v_g.
\]
Expanding the exponential gives
\begin{equation}
  \label{eq:3}
  \Cov_J(\omega_e,\omega_f)
  =\frac{2}{\cZ(J)^2}
  \nlsum_{m\in\mathbb N_0^\cE}
  \left(\prod\nolimits_g\frac{(2J_g)^{m_g}}{m_g!}\right)
  \left[
    \int_{\T^n}\sin u_e\sin u_f\prod\nolimits_g(\cos u_g)^{m_g}\,du
  \right]^2\geq0,
\end{equation}
which completes the proof.
\end{proof}
We are now ready to formally state the desired MTP$_2$ property.
\begin{prop}\label{prop:radial}
Let \(A\) be a symmetric inverse \(M\)-matrix, and let \(\zeta\in\mathbb C^n\) %
be a complex Gaussian vector with covariance \(A\). Then \(X=(|\zeta_1|^2,\ldots,|\zeta_n|^2)\) has an MTP\(_2\) density on \((0,\infty)^n\).
\end{prop}

\begin{proof}
By~\eqref{eq:density} and the derivatives of \(\log\cZ\) above, for \(i\ne j\),
\begin{align*}
 \partial_i\partial_j\log p(x)
 &=\nlsum_e\E_J[\omega_e]\,\partial_i\partial_jJ_e(x)+\nlsum_{e,f}\Cov_J(\omega_e,\omega_f)
          \,\partial_iJ_e(x)\,\partial_jJ_f(x)\geq0.
\end{align*}
Indeed, all first derivatives of the couplings are nonnegative, and their mixed derivatives at distinct coordinates vanish except for \(e=\{i,j\}\), where
\[
 \partial_{ij}J_e(x)=\frac{J_e(x)}{4x_ix_j}\geq0.
\]
Lemma~\ref{lem:angular} supplies the remaining nonnegative factors.
\end{proof}

\subsection{Step 4, squarefree moments} The MTP\(_2\) property proves the Koteljanskii inequality~\eqref{eq:lsm} after a change of measure, as shown below. %

\begin{lemma}\label{lem:moments}
Suppose that a random vector \(X\in(0,\infty)^n\) has an MTP\(_2\) density and finite squarefree moments. Then, the set function \(f(S)=\E\prod_{i\in S}X_i\) is log-supermodular.
\end{lemma}
\begin{proof}
Fix \(S,T\) and put \(R=S\cap T\). The density $p_R(x)=f(R)^{-1}\prod_{i\in R}x_ip(x)$
is MTP\(_2\), since the multiplying factor is a product of functions of individual coordinates. By the FKG inequality for MTP\(_2\) densities \citep{karlin1980classes}, increasing functions are positively correlated under this density. Apply this rule to the increasing functions \(u(x)=\prod_{i\in S\setminus R}x_i\) and \(v(x)=\prod_{i\in T\setminus R}x_i\), first truncated at a level \(N\) and then with \(N\to\infty\) by monotone convergence; the assumed moments make all terms finite. It yields
\[
 \frac{f(S\cup T)}{f(R)}
 \geq \frac{f(S)}{f(R)}\frac{f(T)}{f(R)},
\]
which proves the assertion.
\end{proof}

\noindent Now we have all the steps, and can now state the final proof.
\begin{proof}[\textbf{Proof of Theorem~\ref{thm:main}}]
By Step~1 we may assume that \(A\) is symmetric. By Step~3's Proposition~\ref{prop:radial}, the vector \(X\) has an MTP\(_2\) density, while by Step~2 its squarefree moments are \(\per(A_S)\). Thus, using Step~4's Lemma~\ref{lem:moments} we may immediately conclude~\eqref{eq:lsm}.
\end{proof}

\begin{rmk}[Repeated indices]\label{rmk:mixed}
The same proof gives log-supermodularity of all mixed moments of \(X\) on \(\mathbb N_0^n\): tilt by \(x^{a\wedge b}\) and apply the FKG inequality to \(x^{a-a\wedge b}\) and \(x^{b-a\wedge b}\). By the same moment theorem \citep{reed1962moment}, these moments are the permanents of the matrices obtained from \(A\) by repeating row and column \(i\) the same number of times.
\end{rmk}

\begin{rmk}[Positive association is not enough]\label{rmk:assoc}
The tilt in the proof of Lemma~\ref{lem:moments} cannot be skipped: positive association of a random vector does not make its squarefree moments log-supermodular. For a random vector \(W\in\reals_{\ge0}^n\) write \(\Phi_W(S)=\E\prod_{i\in S}W_i\). For independent \(N\sim\operatorname{Poisson}(\lambda)\) and \(N'\sim\operatorname{Poisson}(\mu)\), a constant \(a\ge0\), and \(W=(a+N,a+N+N',a+N')\), a direct moment calculation gives
\[
 \Phi_W(\{1,2,3\})\Phi_W(\{2\})-\Phi_W(\{1,2\})\Phi_W(\{2,3\})=-\lambda\mu.
\]
The vector \(W\) is positively associated, since its coordinates are nondecreasing functions of the independent variables \(N\) and \(N'\) \citep{esary1967association}.

The same example, with \(a=0\) and \(\lambda=\mu=1\), shows that disjoint convolution does not preserve log-supermodularity. Here \(\Phi_W(\{1,2,3\})\Phi_W(\{2\})=4\cdot2<3\cdot3=\Phi_W(\{1,2\})\Phi_W(\{2,3\})\). Since \(W\) is the sum of the independent vectors \((N,N,0)\) and \((0,N',N')\), expanding \(\prod_{i\in S}W_i\) over the two summands shows that \(\Phi_W\) is the disjoint convolution
\[
  (f\dotast g)(S):=\sum_{R\subseteq S}f(R)g(S\setminus R)
\]
of their moment functions \(f\) and \(g\). Both are log-supermodular: \(f\) vanishes outside \(2^{\{1,2\}}\), where the only nontrivial instance is \(f(\{1,2\})f(\varnothing)=2\ge1=f(\{1\})f(\{2\})\), and similarly for \(g\) on \(2^{\{2,3\}}\). The example therefore contradicts Corollary~11 of the technical report of \citet{lovasz2003}, which asserts that the disjoint convolution of two log-supermodular functions is log-supermodular.
\end{rmk}

\section{Counterexamples}\label{sec:counter}
This section proves Theorem~\ref{thm:counterexample}; it also briefly considers the more general setup of \(\alpha\)-permanents, where a similar negative result holds (Theorem~\ref{thm:allalpha}). Both results come from inverse \(M\)-matrices whose inverses have the directed graph of two cycles that share a path; the matrix of Theorem~\ref{thm:counterexample} belongs, after a column scaling, to the family used for Theorem~\ref{thm:allalpha} (Remark~\ref{rmk:family}).

\subsection{Proof of Theorem~\ref{thm:counterexample}}
\begin{proof}[\textbf{Proof of Theorem~\ref{thm:counterexample}}]
Direct multiplication gives
\[
 A^{-1}=\frac1{84}\begin{pmatrix}
 7&0&-7&0&0\\
 0&7&-7&0&0\\
 0&0&14&-14&0\\
 0&0&0&14&-14\\
 -1&-1&0&0&14
 \end{pmatrix}.
\]
This matrix has nonpositive off-diagonal entries and inverse \(A>0\). The matrix \(A\) is not symmetrizable: if \(\Delta A\Delta^{-1}\) were symmetric, so would be its inverse \(\Delta A^{-1}\Delta^{-1}\), but a diagonal similarity preserves the zero pattern of \(A^{-1}\), and \((A^{-1})_{13}\ne0=(A^{-1})_{31}\). Take \(S=\{1,3,4,5\}\), \(T=\{2,3,4,5\}\), and \(R=\{3,4,5\}\). The four permanents are
\[
 \per(A)=182784,\qquad \per(A_R)=546,\qquad
 \per(A_S)=\per(A_T)=9996.
\]
Consequently,
\[
 \per(A)\per(A_R)=99800064
 <99920016=\per(A_S)\per(A_T),
\]
and their ratio is \(832/833\).
\end{proof}

\subsection{\texorpdfstring{\(\alpha\)}{alpha}-permanents}\label{sec:alpha}
While developing Theorems~\ref{thm:main} and~\ref{thm:counterexample} we wondered whether similar results might hold for the \(\alpha\)-permanents of \citet{verejones1997alpha}. The situation here turns out to be more nuanced, as noted in Remark~\ref{rmk:real} below.

For \(\alpha>0\), the \(\alpha\)-permanent of a matrix \(H\) indexed by \(S\) is
\begin{equation}\label{eq:alphaper}
  \per_\alpha(H)=\sum_{\pi\in\mathfrak S(S)}\alpha^{c(\pi)}\prod_{i\in S}H_{i,\pi(i)},
\end{equation}
where \(\mathfrak S(S)\) is the group of permutations of \(S\) and \(c(\pi)\) is the number of cycles of \(\pi\); the ordinary permanent is \(\alpha=1\). A random vector \(Y\in[0,\infty)^n\) with Laplace transform \(\det(I+A\Lambda)^{-\alpha}\), where \(\Lambda=\diag(s_1,\dots,s_n)\), is an \(\alpha\)-permanental vector with kernel \(A\) in the sense of \citet{verejones1997alpha}, and its squarefree moments are the principal \(\alpha\)-permanents,
\begin{equation}\label{eq:vj}
  \E\prod_{i\in S}Y_i=\per_\alpha(A_S),\qquad S\subseteq[n],
\end{equation}
by \citet[Proposition~4.2]{verejones1997alpha}. \citet{eisenbaum2009permanental} proved that an infinitely divisible permanental vector with kernel \(A\) exists if and only if \(A\) is, up to diagonal similarity, the Green function of a transient Markov chain, that is, an inverse \(M\)-matrix.

\begin{rmk}[Positive cases]\label{rmk:real}
For symmetric positive definite \(A\) there is a real Gaussian analogue. If \(\eta\sim N(0,A)\), then \((\eta_i^2)_{i}\) is MTP\(_2\) exactly when \(A^{-1}\) is, up to a signature, an \(M\)-matrix \citep{karlin1981total}. The vector \((\eta_i^2/2)_i\) has Laplace transform \(\det(I+A\Lambda)^{-1/2}\), so by~\eqref{eq:vj} its squarefree moments are \(\per_{1/2}(A_S)\), and \(\E\prod_{i\in S}\eta_i^2=2^{|S|}\per_{1/2}(A_S)\). Lemma~\ref{lem:moments} and the diagonal similarity argument above therefore give Theorem~\ref{thm:main} for \(\per_{1/2}\), but not for the permanent itself.
If \(A\) is a symmetric inverse \(M\)-matrix and the graph of \(A^{-1}\) is a tree, \citet[Section~4]{royen2025improved} proves that the \(\alpha\)-permanental vector with kernel \(A\) has an MTP\(_2\) density for every \(\alpha>0\); formula~\eqref{eq:vj} and Lemma~\ref{lem:moments} then give log-supermodularity of \(S\mapsto\per_\alpha(A_S)\) for every \(\alpha>0\).
\end{rmk}

The Gaussian argument uses \(\alpha=1\) (complex) and \(\alpha=\tfrac12\) (real) in an essential way, and for other \(\alpha\) the symmetrizable case is open when the graph of \(A^{-1}\) has cycles (Section~\ref{sec:open}). For general inverse \(M\)-matrices the answer is negative. Every inverse \(M\)-matrix is the kernel of an \(\alpha\)-permanental vector for every \(\alpha>0\) \citep{eisenbaum2009permanental}, so one may ask whether some other order \(\alpha\) gives the analogue of~\eqref{eq:lsm} for all inverse \(M\)-matrices. None does---see Appendix~\ref{app:alpha} for details.

\section{Open problems}\label{sec:open}

\begin{enumerate}[label=(\alph*),itemsep=1pt]
  \item Characterize the inverse \(M\)-matrices whose principal permanents are log-supermodular. By Theorem~\ref{thm:main} this class contains the symmetrizable ones, and by Theorem~\ref{thm:counterexample} it is smaller than the class of all inverse \(M\)-matrices.
  \item For symmetrizable inverse \(M\)-matrices and \(\alpha\notin\{\tfrac12,1\}\), decide whether \(S\mapsto\per_\alpha(A_S)\) is log-supermodular. Theorem~\ref{thm:main} and Remark~\ref{rmk:real} settle \(\alpha=1\), \(\alpha=\tfrac12\), and every \(\alpha>0\) when the graph of \(A^{-1}\) is a tree.
  \item The proof of Theorem~\ref{thm:allalpha} needs the dimension to grow with \(\alpha\). For fixed \(n\), is there an \(\alpha_n\) such that \(S\mapsto\per_\alpha(A_S)\) is log-supermodular for every \(n\times n\) inverse \(M\)-matrix and every \(\alpha\ge\alpha_n\)?
  \item For an inverse \(M\)-matrix \(A\), consider
  \[
    f_A(x_0,x)=\nlsum_{S\subseteq[n]}\per(A_S)^{-1}x_0^{\,n-|S|}\prod\nolimits_{i\in S}x_i.
  \]
  For positive semidefinite \(A\), the polynomial with coefficients \(\det(A_S)\) in place of \(1/\per(A_S)\) is \(\det(x_0I+\diag(x)A)\), which is real stable \citep{branden08} and hence Lorentzian \citep{branden2020lorentzian}. Is \(f_A\) Lorentzian for every inverse \(M\)-matrix \(A\)? Numerical tests with \(n\le5\) found no counterexample, including the matrix of Theorem~\ref{thm:counterexample}; for random nonnegative \(A\), by contrast, \(f_A\) is usually not Lorentzian, and for symmetric inverse \(M\)-matrices the polynomial with coefficients \(\per(A_S)\) was never Lorentzian.
\end{enumerate}

\section*{AI Usage} The conjecture was formulated by the author in April 2015, with partial proofs. After a sequence of back and forth discussions with GPT-5.5 (Pro) the first proof was attained. Subsequently, an audit with GPT-6 (Astra) revealed a few small errors, that were fixed. AI assisted with the drafting of the paper (GPT-6, and Claude Opus 5.5 and Fable 5.1), while the author strove to rewrite the entire paper by hand to simplify, polish, and compactify.

\section*{Acknowledgements}
The author (SS) thanks the Alexander von Humboldt (AvH) Foundation for
their generous support via a Humboldt Professorship in AI.

\appendix
\section{More on $\alpha$-permanents}
\label{app:alpha}
\begin{theorem}\label{thm:allalpha}
For every \(\alpha>0\) there are an integer \(n\), an entrywise positive \(n\times n\) inverse \(M\)-matrix \(A\), and sets \(S,T\subseteq[n]\) such that
\[
 \per_\alpha(A_{S\cup T})\per_\alpha(A_{S\cap T})
 <\per_\alpha(A_S)\per_\alpha(A_T).
\]
\end{theorem}

\begin{proof}[\textbf{Proof of Theorem~\ref{thm:allalpha}}]
Fix \(k\geq1\) and take the \(n=k+2\) vertices \(1,2,r_1,\ldots,r_k\). Let \(B\) have weight one on the edges \(1\to r_1\), \(2\to r_1\), and \(r_j\to r_{j+1}\), and weight \(t\) on \(r_k\to1\) and \(r_k\to2\), with \(0<t<1/2\). All other entries vanish. The only directed cycles are \(1\to r_1\to\cdots\to r_k\to1\) and \(2\to r_1\to\cdots\to r_k\to2\), and they share \(r_1,\dots,r_k\), so the permutation expansion of the determinant gives, for \(U=\diag(u)\),
\[
 \det(I-BU)=1-t(u_1+u_2)\prod\nolimits_{j=1}^ku_{r_j}.
\]
Taking \(U=sI\) shows that the nonzero eigenvalues of \(B\) are the \((k+1)\)st roots of \(2t\), so \(\rho(B)=(2t)^{1/(k+1)}<1\). The graph of \(B\) is strongly connected, so \(A=(I-B)^{-1}=\sum_{m\ge0}B^m\) is entrywise positive, and it is an inverse \(M\)-matrix.

Let \((N_1,N_2)\) have the negative multinomial law with probability generating function
\[
 \E[v^{N_1}w^{N_2}]
 =\left(\frac{1-2t}{1-tv-tw}\right)^\alpha,
\]
put \(L=N_1+N_2\), and let \(W=(N_1,N_2,L,\ldots,L)\), with coordinates indexed by \(1,2,r_1,\ldots,r_k\). Conditional on \(W\), let \(Y_1,\dots,Y_n\) be independent gamma variables with shape \(\alpha+W_i\) and rate one, so that \(\E[e^{-s_iY_i}\mid W]=q_i^{\alpha+W_i}\) with \(q_i=1/(1+s_i)\). For \(s\in[0,\infty)^n\) put \(Q=\diag(q_1,\dots,q_n)\) and \(P=\prod_{j=1}^kq_{r_j}\). Since \(q^W:=\prod_iq_i^{W_i}=(q_1P)^{N_1}(q_2P)^{N_2}\),
\[
 \E e^{-\sum_is_iY_i}
 =\Bigl(\prod\nolimits_iq_i\Bigr)^{\alpha}\left(\frac{1-2t}{1-t(q_1+q_2)P}\right)^\alpha
 =\left(\frac{\det(I-B)}{\det(I+\Lambda)\det(I-BQ)}\right)^\alpha,
\]
by the formula for \(\det(I-BU)\) with \(U=Q\) and with \(U=I\). Here \(\Lambda=\diag(s)\), and \((I-BQ)(I+\Lambda)=I-B+\Lambda\), so the last denominator is \(\det(I-B+\Lambda)=\det(I-B)\det(I+A\Lambda)\). Hence \(Y\) has Laplace transform \(\det(I+A\Lambda)^{-\alpha}\). The variable \(L\) is negative binomial, so \(W\) and \(Y\) have finite moments of all orders, and~\eqref{eq:vj}, together with the conditional independence of the \(Y_i\) and \(\E[Y_i\mid W]=\alpha+W_i\), gives
\[
 \per_\alpha(A_S)=\E\prod\nolimits_{i\in S}(\alpha+W_i),\qquad S\subseteq[n].
\]

Set \(R=\{r_1,\ldots,r_k\}\), \(S=R\cup\{1\}\), and \(T=R\cup\{2\}\), so that \(S\cup T=[n]\) and \(S\cap T=R\). Tilt the law of \((N_1,N_2)\) by \((\alpha+L)^k\), and denote the tilted expectation by \(\E_*\). Then
\[
 \per_\alpha(A_{S\cup T})\per_\alpha(A_R)-\per_\alpha(A_S)\per_\alpha(A_T)
 =\bigl(\E(\alpha+L)^k\bigr)^2\Cov_*(N_1,N_2).
\]
The generating function of \((N_1,N_2)\) depends on \(v,w\) only through \(v+w\), so given \(L\) the variable \(N_1\) is binomial with parameters \((L,1/2)\), and \(N_2=L-N_1\). This remains so after the tilt, which depends only on \(L\). Hence
\[
 \Cov_*(N_1,N_2)=\tfrac14(\Var_*(L)-\E_*L).
\]

Write \(z=2t\), and let \((\alpha)_m\) denote the rising factorial. The variable \(L\) has \(\mathbb P(L=m)=(1-z)^\alpha(\alpha)_mz^m/m!\), so the tilted law is \(\mathbb P_*(L=m)=a_mz^m/G(z)\), where
\[
 G(z)=\sum_{m\geq0}a_mz^m,\qquad
 a_m=\frac{(\alpha)_m}{m!}(\alpha+m)^k.
\]
The series converges for \(|z|<1\). Since \(\E_*L=zG'(z)/G(z)\) and \(\E_*[L(L-1)]=z^2G''(z)/G(z)\),
\[
 \Var_*(L)-\E_*L=z^2(\log G)''(z).
\]
At zero the numerator of \((\log G)''\) is
\begin{align*}
 G(0)G''(0)-G'(0)^2
 &=\alpha^{k+1}(\alpha+1)(\alpha+2)^k
   -\alpha^2(\alpha+1)^{2k}.
\end{align*}
This quantity is negative exactly when
\[
 \left(\frac{\alpha(\alpha+2)}{(\alpha+1)^2}\right)^k
 <\frac{\alpha}{\alpha+1}.
\]
The base on the left lies strictly between zero and one, so a sufficiently large integer \(k\) satisfies this inequality. Continuity then makes \((\log G)''(2t)<0\) for sufficiently small positive \(t\), and the covariance above is negative.
\end{proof}

\begin{rmk}\label{rmk:family}
For \(\alpha=1\) one may take \(k=3\), because \((3/4)^3<1/2\). With the vertices ordered \(1,2,r_1,r_2,r_3\) and \(t=1/14\), the matrix of Theorem~\ref{thm:counterexample} is \(84\,(I-B)^{-1}\diag(\tfrac17,\tfrac17,\tfrac1{14},\tfrac1{14},\tfrac1{14})\); a positive column scaling multiplies \(\per(A_S)\) by a modular factor and so does not affect~\eqref{eq:lsm}. For \(k=2\) the sign at zero is that of \(\alpha^2+\alpha-1\), so the four-vertex members of the family violate log-supermodularity for small \(t\) exactly when \(\alpha<(\sqrt5-1)/2\).
\end{rmk}

\end{document}